\documentclass[12pt]{amsart}
\usepackage[matrix,arrow,curve,frame]{xy}
\usepackage{amsmath,amsthm,amssymb,enumerate}
\usepackage{latexsym}
\usepackage{amscd}
\usepackage[colorlinks=false]{hyperref}
\usepackage{euscript}

\usepackage{esint}
\usepackage{bm}
\usepackage{mathtools}
\usepackage{extarrows}

\numberwithin{equation}{section}
\numberwithin{figure}{section}

\usepackage{amsthm}

\newtheorem{theorem}{Theorem}
\newtheorem{corollary}{Corollary}
\newtheorem{lemma}{Lemma}
\newtheorem{definition}{Definition}

\newtheorem{proposition}{Proposition}

\numberwithin{theorem}{section}
\numberwithin{lemma}{section}
\numberwithin{corollary}{section}
\numberwithin{definition}{section}
\numberwithin{example}{section}
\numberwithin{proposition}{section}

\def\cA{{\mathcal A}}

\def\cE{{\mathcal E}}

\DeclareMathOperator{\im}{im}

\DeclareMathOperator{\Sym}{Sym}
\DeclareMathOperator{\End}{End}
\DeclareMathOperator{\Ker}{Ker}

\DeclareMathOperator{\Ima}{Im}
\DeclareMathOperator{\PSL}{PSL}

\newfont{\german}{eufm10}

\title{\textbf{The Global Sections of Chiral de Rham Complexes on Closed Complex Curves}}

\begin{document}
\pagestyle{plain}
\title{The Global Sections of Chiral de Rham Complexes on Closed Complex Curves}
\author{Bailin Song}
\address{School of Mathematical Sciences, University of Science and Technology of China, Hefei, 230026, People's Republic of China}
\email{bailinso@ustc.edu.cn}
\thanks{B. Song is supported  by National Natural Science Foundation of China Grant \#12171447.}	
\author{Wujie Xie}
\address{School of Mathematical Sciences, University of Science and Technology of China, Hefei, 230026, People's Republic of China}
\email{xiewujie@mail.ustc.edu.cn}

\begin{abstract}
	The space of global sections of the chiral de Rham complex on any closed complex curve with genus $g \ge2$ is calculated.
\end{abstract}	

\maketitle 
	\section{Introduction}
	The chiral de Rham complex $\Omega_{X}^{ch}$ is a sheaf of vertex superalgebras introduced  by Malikov, Schechtman and Vaintrob in \cite{MSV}, which can be defined on either smooth manifolds, complex analytic manifolds or nonsingular algebraic varieties. It is graded by conformal weight and its weight zero part coincides with the ordinary de Rham sheaf. This sheaf  has attracted
	significant attention in both the physics and mathematics literature. For example, in \cite{Ka}, Kapustin has shown that if X is a Calabi-Yau manifold, the cohomology of this sheaf can be identified with the infinite-volume limit of the half-twisted sigma model defined by E. Witten;  in \cite{Bo}, Borisov presented an approach to  toric Mirror Symmetry by the chiral de Rham complex. 
	
	The space of global
	sections $\Gamma(X, \Omega_X^{ch})$ is always a vertex superalgebra, and it is known to have extra structure
	when $X$ is endowed with geometric structures. For example, when $X$ is Kähler or hyperkähler,
it has an $\mathcal N = 2$ structure or $\mathcal N = 4$ structure, respectively \cite{BHS}. Recently in \cite{S3, LS}, Linshaw and the first author gave a complete description of $\Gamma(X, \Omega_X^{ch})$ for any compact Ricci-flat Kähler manifold $X$. In \cite{D}, for any congruence subgroup $G \subset S L (2, \mathbb{R})$,
Dai constructed a basis of the $G$-invariant global sections of the chiral de Rham complex on the upper half plane, which are holomorphic at the cusps. The vertex operations are determined by a modification of the Rankin-Cohen brackets of modular forms. Further in \cite{DS}, it is shown that if the global sections are allowed to be meromorphic at cusp, the space of $G$-invariant global sections will be a simple vertex algebra. There have been some descriptions about the structure of $H^0(X, \Omega^{ch}_X)$ when $X$ is a closed complex curve with genus $0$ or $1$. In the case $g=0$, $X$ is of constant positive curvature and a description of $H^0(X, \Omega^{ch}_X)$ through $\hat{sl_2}$-modules was given in \cite{MS2}. In the case $g=1$, $X$ is of $0$ curvature and $H^0(X, \Omega^{ch}_X)$ is just a $\beta\gamma-bc$ system. The case $g\ge 2$ has never been calculated yet, with the curvature of $X$ being constant negative.

When $X$ is a compact Kähler manifold, a filtration $(\mathcal{F}_{s})_{s\in \mathbb{Z}}$ of $H^0(X, \Omega^{ch}_X)$ was in fact constructed in \cite{S3}, such that $\bigoplus_{s}\mathcal{F}_{s}/\mathcal{F}_{s-1}$ is isomorphic to the space of holomorphic sections of 
 $SW(\bar{T}^*X)$, where $SW(\bar{T}^*X)$ is isomorphic to $$\Sym^*(\bigoplus^{\infty}_{n=1}\bar{T}X)\bigotimes \Sym^*(\bigoplus^{\infty}_{n=1}\bar{T}^*X)\bigotimes \wedge^*(\bigoplus^{\infty}_{n=1}\bar{T}X)\bigotimes \wedge^* (\bigoplus^{\infty}_{n=1}\bar{T}^*X)$$as antiholomorphic vector bundles. Here $\bar{T}X$ and $\bar{T}^*X$ are antiholomorphic tangent and cotangent bundle of $X$, respectively. In this way $H^0(X, \Omega_X^{ch})$ is described in terms of holomorphic sections of $SW(\bar{T}^*X)$. More technically, in \cite {S3}, a soft resolution $(\Omega^{ch,*}_X, \bar{\partial})$ of the chiral de Rham complex $\Omega_X^{ch}$ was constructed, then an isomorphism of complexes of sheaves from $(\Omega^{0,*}_X(SW(\bar{T}^*X))(X), \bar{D})$ to $(\Omega^{ch,*}_X, \bar{\partial})$ was constructed, where $\bar{D}$ is an elliptic operator with $\bar{D}=\bar{\partial}^{\prime}+F_1+F_2+\ldots$ in which $\bar{\partial}^{\prime}$ is the antiholomorphic part of the Chern connection and $F_i$'s are first order differential operators. In this formalism, the calculation of $H^0(X, \Omega_X^{ch})$ is equivalent to the calculation of $H^0(\Omega^{0,*}_X(SW(\bar{T}^*X))(X), \bar{D})$(Theorem \ref{thm:iso2}).

The structure of $H^0(X, \Omega^{ch}_X)$ when $X$ is a compact Ricci-flat Kähler manifold was calculated in [\cite {S3}]. In this paper, we  extend the method developed in \cite {S3} to a class of compact non Ricci-flat Kähler manifolds, that is, the closed locally symmetric spaces, and calculate the space of global sections of chiral de Rham complex on any closed curves with genus $g\geq 2$ (Theorem \ref{thm:global}).

 The paper is organized as follows:  in section \ref{section 2},  we introduce the chiral de Rham complex; in section \ref{section3},  we introduce the method of calculating the global sections of chiral de Rham complex in \cite{S3}  and apply the method to Hermitian locally symmetric spaces. In section \ref{section4}, we  calculate the space of global sections for any closed complex curve $X$ with genus $g \geq 2$.
	
	\section{Chiral de Rham Complex}\label{section 2}
	\subsection{Vertex algebras}
	In this paper, we will follow the formalism of vertex algebras developed in \cite{Kac}. 
	
	A formal distribution is the following formal expressions
	$$\sum_{m,n,\ldots \in \mathbb{Z}}a_{m,n,\ldots}z^m w^n\ldots,$$
	where $a_{m,n,\ldots}$ are elements of a vector space $U$ over $\mathbb{C}$, and $z, w, \ldots$ are indeterminates. These expressions form a vector space over $\mathbb{C}$ denoted by $U[[z,z^{-1},w,w^{-1},\ldots]]$.
	
	A vertex
	algebra is the data $(\cA, Y, L_{-1}, 1)$. In this notation,
	\begin{enumerate} 
		\item $\cA$ is a $\mathbb Z_2$-graded vector space over $\mathbb C$. The $\mathbb Z_2$-grading is called parity, and $|a|$ denotes the parity of a homogeneous element $ a \in \cA$. 
		\item $Y$ is an even linear map
		$$Y : \cA \to \End (\cA) [[z, z^{-1}]],\qquad Y(a) = a(z) = \sum_{n\in \mathbb Z}a_{(n)}z^{-n-1}.$$ Here $z$ is a formal variable and $a(z)$ is called the field corresponding to $a$.  
		\item $1 \in \cA$ is called the vacuum vector. 
		\item $L_{-1}$ is an even endomorphism of $\cA$.
	\end{enumerate}
	They
	satisfy the following axioms:
	\begin{itemize}
		\item
		\textsl{Vacuum axiom.} $L_{-1}1 =0$; $1(z)= Id$; for $a\in \cA$, $n\geq 0$, $a_{(n)}1=0$ and $a_{(-1)}1=a$;
		\item \textsl{Translation invariance axiom.} For $a\in \cA$, $$[L_{-1}, Y(a)] =\partial a(z);$$
		\item \textsl{Locality axiom.} Let $z,w$ be formal variables. For homogeneous $a, b \in \cA$, $$(z-w)^k [a(z), b(w)]=0$$ for some $k\geq 0$, where $[a(z), b(w)] = a(z) b(w) - (-1)^{|a| |b|} b(w) a(z)$.
	\end{itemize}
	For $a,b\in \cA$, $n\in \mathbb Z_{\geq 0}$, $a_{(n)}b$ is their $n^{\text{th}}$ product and their operator product expansion (OPE) is
	$$a(z)b(w)\sim \sum_{n\geq 0} (a_{(n)}b)(w)(z-w)^{-n-1}.$$
	The Wick product of $a(z)$ and $b(z)$ is
	$:a(z)b(z):  \ =(a_{(-1)}b)(z)$. The other negative products are given by
	$$:\partial^na(z)b(z):\ =n!(a_{(-n-1)}b)(z).$$
	
	Note that for any formal distributions $ a(z) = \sum_{n\in \mathbb Z}a_{(n)}z^{-n-1}$ and $b(w) = \sum_{n\in \mathbb Z}b_{(n)}w^{-n-1}$ which are not necessarily fields of vertex algebras, the product $a_{(n)}b$, $n \in \mathbb{Z}$ can still be defined.

	\subsection{$\beta\gamma-bc$ system.}
	Let $V$ be a $d$-dimensional complex vector space.
	The $\beta\gamma$-system  $\mathcal S(V)$ and $bc$-system $\cE(V)$ were introduced in \cite{FMS}.  A $\beta\gamma-bc$ system $\mathcal{W}(V)=\mathcal S(V)\otimes \mathcal E(V)$ is a vertex algebra generated by even elements $\beta^i$, $\gamma^i$, and odd elements $b^i$, $c^i$, where $i$ ranges from $1$ to $d$. The only nontrivial OPEs among them are:
		\begin{equation*}
			\beta^i(z)\gamma^j(w) \sim \delta_{ij}(z-w)^{-1},\quad 
			b^i(z)c^j(w) \sim \delta_{ij}(z-w)^{-1}.
		\end{equation*}
	
		Let $\mathcal{W}_{+}(V)$ be the subalgebra of $\mathcal{W}(V)$ generated by $\beta^{i}$, $\alpha^{i}=\partial\gamma^{i}$, $b^{i}$ and $c^{i}$. 
	It is a topological vertex algebra with the following four fields:
		\begin{equation}\label{eqn:QLJG}
			\begin{aligned}
				&Q(z)=\sum^{d}_{i=1}:\beta^i(z)c^i(z):,\quad L(z)=\sum^{d}_{i=1}(:\beta^i(z)\partial\gamma^i(z):-:b^i(z)\partial{c^i}(z):),           \\
				&J(z)=-\sum^{d}_{i=1}:b^i(z)c^i(z):,\quad G(z)=\sum^{d}_{i=1}:b^i(z)\partial\gamma^i(z):.
			\end{aligned}
		\end{equation} 
The subalgebra $\mathcal{W}_{+}(V)$ is graded with conformal weights $k$ and fermionic numbers $l$:
		\begin{equation*} 
			\mathcal{W}_{+}(V)=\bigoplus_{k\in \mathbb{Z}_{\geq 0}}\bigoplus_{l\in \mathbb{Z}}\mathcal{W}_{+}(V)[k,l].
		\end{equation*}
		Here $\mathcal{W}_{+}(V)[k,l]=\lbrace A\in \mathcal{W}_{+}(V)|L_{(1)}A=kA, J_{(0)}A=lA\rbrace$.
		
	\subsection{ Chiral de Rham Algebra.}
Let $X$ be a complex manifold of dimension $d$. For a local coordinate system  $(U, \gamma^1, \ldots, \gamma^d)$ of $X$, $\Omega^{ch}_{X}(U)$ is given by:
		\begin{equation*}
			\Omega^{ch}_{X}(U):=\mathcal{W}(\mathbb{C}^{d})\otimes_{\mathbb{C}[\gamma^1_{(-1)}, \ldots, \gamma^d_{(-1)}]}\mathcal{O}(U).
		\end{equation*}
		Here the action of $\gamma^i_{(-1)}$ on the analytic function ring $\mathcal{O}(U)$ is identified with the product of $\gamma^i$.
	
		Then $\Omega^{ch}_{X}(U)$ is a vertex algebra generated by $\beta^i(z)$, $b^i(z)$, $c^i(z)$, and $f(z)$, $f \in \mathcal{O}(U)$, with relations:
		\begin{equation*}
			:f(z)g(z):=fg(z),\quad f,g \in \mathcal{O}(U).
		\end{equation*} 
		The only nontrivial OPEs among these generators are:
		\begin{equation*}
			\beta^i(z)f(w)\sim \dfrac{\frac{\partial f}{\partial \gamma^i}(z)}{z-w},\quad
			b^i(z)c^j(w)\sim \dfrac{\delta_{ij}}{z-w}.
		\end{equation*}
	
	Let $\tilde{\gamma}^1, \ldots, \tilde{\gamma}^d$ be a new coordinate system on $U$ with the relations:
	\begin{equation*}
		\tilde{\gamma}^i=f^i(\gamma^1, \dots, \gamma^d), \quad \gamma^i=g^i(\tilde{\gamma}^1, \ldots, \tilde{\gamma}^d).
	\end{equation*}
	The following coordinate change equations is considered in \cite{MSV}:
	\begin{equation*}
		\begin{aligned}
			&\partial\tilde{\gamma}^i(z)=\sum^{d}_{j=1}:\dfrac{\partial f^i}{\partial \gamma^j}\partial\gamma^j(z):, \\
			&\tilde{b}^i(z)=\sum^{d}_{j=1}:\dfrac{\partial g^j}{\partial \tilde{\gamma}^i}(f(\gamma))(z)b^j(z):,   \\
			&\tilde{c}^i(z)=\sum^{d}_{j=1}:\dfrac{\partial f^i}{\partial \gamma^j}c^j(z):,   \\
			&\tilde{\beta}^i(z)=\sum^d_{j=1}(:\dfrac{\partial g^j}{\partial \tilde{\gamma}^i}(f(\gamma))(z)\beta^j(z):+\sum^d_{k=1}::\dfrac{\partial}{\partial \gamma^k}(\dfrac{\partial g^j}{\partial \tilde{\gamma}^i}(f(\gamma)))(z)c^k(z):b^j(z):)
		\end{aligned}
	\end{equation*}
		and for a general complex manifold $X$, the sections $L(z)$, $G(z)$, $Q_{(0)}$ and $J_{(0)}$ are globally defined.
	
		The complex $\Omega^{ch}_{X}$ is graded as a sheaf of vertex algebras with conformal weights $k$ and fermion numbers $l$:
		\begin{equation*}
			\Omega^{ch}_{X}=\bigoplus_{k\in \mathbb{Z}_{\geq 0}}\bigoplus_{l\in \mathbb{Z}}\Omega^{ch}_{X}[k, l].
		\end{equation*}
	Here $\Omega^{ch}_{X}[k, l](U)=\lbrace A \in \Omega^{ch}_{X}(U)|L_{(1)}A=kA, J_{(0)}A=lA\rbrace$ for any open subset $U$ of $X$.
	
	Using the above construction, we obtain a smooth version of the chiral de Rham complex  $\Omega_X^{ch,sm}$ on $X$, by regarding $X$ as a smooth real manifold, and replacing $\mathcal W(\mathbb C^d)$, the complex coordinate system and the space of analytic functions $\mathcal{O}(U)$ by $\mathcal W(\mathbb C^{2d})$, the smooth coordinate system and the space of smooth complex functions $C^{\infty}(U)$, respectively.

	The sheaf $\Omega_X^{ch, sm}$ includes $\Omega_X^{ch}$ and its complex conjugate. 
The sheaf $\Omega_X^{ch,sm}$ is graded by: $$\Omega^{ch,sm}_X=\bigoplus_{k\geq 0}\bigoplus_{l}\Omega^{ch,sm}_{X}[k,l].$$ Here $\Omega^{ch, sm}_{X}[k,l](U)=\lbrace A \in \Omega^{ch, sm}_{X}(U)|\bar{L}_{(1)}A=kA, \bar{J}_{(0)}A=lA\rbrace$ for any open subset $U$ of $X$. 
Let $\Omega^{ch,l}_X=\Omega^{ch,sm}_X[0,l]$ and $\Omega^{ch,*}_X=\bigoplus_{l=0}^{d}\Omega^{ch,sm}_X[0,l]$. The operator $\bar{\partial}$ is the restriction of $\bar{Q}_{(0)}$ on $\Omega^{ch,*}_X$.

	
	The following results in \cite{S3} will be used in the calculation of the space of global sections of the chiral de Rham complex on closed curves later:
	\begin{theorem}\label{thm:iso0}
		(\romannumeral1) The chiral Hodge cohomology $H^{*}(X,\Omega^{ch}_X)$ is equal to the cohomology of the complex $(\Omega^{ch,*}_X,\bar{\partial})$. 
		
		(\romannumeral2) The chiral Hodge cohomology $H^{*}(X,\Omega^{ch}_X)$ is a vertex algebra. In particular, $\Gamma(X, \Omega^{ch}_{X})=H^{0}(X, \Omega^{ch}_{X})$ is a sub vertex algebra of $H^{*}(X, \Omega^{ch}_{X})$.
	\end{theorem}
	
	\section{Antiholomorphic vector bundles corresponding to chiral de Rham complexes}\label{section3}
	
	Let $X$ be a compact K\"{a}hler manifold with a K\"{a}hler form $h$. $\bar{T}^{*}X$ is the antiholomorphic cotangent bundle of $X$ and $\bar{T}^{*}_{x}X$ is its fiber on $x \in X$.
	
	\subsection{ Vector bundles induced from chiral de Rham complex.}\label{sec:3.1}
	
	In this section we recall the construction of antiholomorphic vector bundles corresponding to chiral de Rham complexes in \cite{S3}. For a holomorphic local coordinate system ($U$, $\gamma^1, \ldots, \gamma^d$), $h$ is locally expressed as $h|_{U}=\sum_{i,j}H_{ij}d\gamma^i\wedge d\bar{\gamma}^j$. Let $(H^{jk})$ be the inverse of $(H_{ij})$, i.e., $\sum_{j}H_{ij}H^{jk}=\delta_{ik}$. Define the generators:
	\begin{equation*}
		\bm{\Gamma}^j=\sum_{i}:H_{ij}\partial\gamma^i:,\quad \bm{c}^j=\sum_{i}:H_{ij}c^i:,\quad \bm{b}^j=\sum_{i}:H^{ji}b^i:,\quad \bm{B}^j=Q_{(0)}\bm{b}^j.
	\end{equation*}
	The only nontrivial OPEs among them are:
	\begin{equation*}
		\bm{B}^i(z)\bm{\Gamma}^j(w)\sim\delta_{ij}(z-w)^{-2},~~\bm{b}^i(z)\bm{c}^j(w)\sim\delta_{ij}(z-w)^{-1}.
	\end{equation*}
	The vertex algebra generated by these elements is denoted by $\mathcal{W}^{\gamma}_{+}$, and it is isomorphic to $\mathcal{W}_{+}(V)$ as a vertex algebra by mapping $\bm{B}^i$, $\bm{\Gamma}^i$, $\bm{b}^i$ and $\bm{c}^i$ to $\beta^i$, $\alpha^i$, $b^i$ and $c^i$, respectively. $\bm{b^{i}}$, $\bm{c}^{i}$ and $\bm{\Gamma}^i$ commute with smooth functions on $U$, and
	\begin{equation*}
		\bm{B}^j(z)f(w)\sim\dfrac{\sum_{i}H^{ji}\frac{\partial f}{\partial \gamma^i}}{z-w}.
	\end{equation*}
	
	Let $\tilde{\gamma}^1, \ldots, \tilde{\gamma}^d$ be the new holomorphic local coordinates on $\tilde{U}$ with the relations $\gamma^i=g_i(\tilde{\gamma})$ and $\tilde{\gamma}^i=f_i(\gamma)$, the K\"{a}hler form is expressed locally as $h|_{\tilde{U}}=\sum_{i,j}\tilde{H}_{ij}d\tilde{\gamma}^i\wedge d\tilde{\gamma}^j$ in the new coordinate system, where $(\tilde{H}^{jk})$ is the inverse of $(\tilde{H}_{ij})$. Let
	\begin{equation*}
		\tilde{\bm{\Gamma}}^j=\sum_{i}:\tilde{H}_{ij}\partial\tilde{\gamma}^i:,\quad \tilde{\bm{c}}^j=\sum_{i}:\tilde{H}_{ij}\tilde{c}^i:,\quad \tilde{\bm{b}}^j=\sum_{i}:\tilde{H}^{ji}\tilde{b}^i:,\quad \tilde{\bm{B}}^j=Q_{(0)}\tilde{\bm{b}}^j.
	\end{equation*}
	Then the following relations hold on $U \cap \tilde{U}$:
	\begin{equation}\label{eqn:transform}
		\tilde{\bm{\Gamma}}^j=\sum_{i}:\overline{\frac{\partial g_{i}}{\partial \tilde{\gamma}^j}}\bm{\Gamma}^i:,\quad \tilde{\bm{c}}^j=\sum_{i}:\overline{\frac{\partial g_{i}}{\partial \tilde{\gamma}^j}}\bm{c}^i:,\quad \tilde{\bm{b}}^j=\sum_{i}:\overline{\frac{\partial f_{j}}{\partial \gamma^i}}\bm{b}^i:,\quad \tilde{\bm{B}}^j=\sum_{i}:\overline{\frac{\partial f_{j}}{\partial \gamma^i}}\bm{B}^i:.
	\end{equation}
	So under the change of coordinates, both $\bm{B}^i$ and $\bm{b}^i$ change like $d\overline{\gamma}^i$, and both $\bm{\Gamma}^i$ and $\bm{c}^i$ change like $\frac{\partial}{\partial \bar{\gamma}^i}$.
	
	Let $SW^{\gamma}$ be the polynomial ring generated by $\bm{b}^i_{(n)}$, $\bm{c}^i_{(n)}$, $\bm{\Gamma}^i_{(n)}$, $\bm{B}^i_{(n)}$, $1\leq i \leq d$, $n<0$, and $S^{\gamma}$ be the set of monomials in $SW^{\gamma}$. Let $SW(\bar{T}^{*}X)$ be an antiholomorphic vector bundle with fibers $SW(\bar{T}_{x}^{*}X)\cong SW^{\gamma}$ on any $x \in U$. 
	
	The set of monomials $S^{\gamma}$ is graded by:
	\begin{equation*}
		\begin{aligned}
			&S^{\gamma}[k, l]=\lbrace A\in S^{\gamma}|[L_{(1)}, A]=kA, [J_{(0)}, A]=lA \rbrace,    \\
			&S^{\gamma}[k, l, s]=\lbrace A\in S^{\gamma}[k,l]|\text{the number of $\bm{B}$'s in $A$ $ - $ the number of $\bm{\Gamma}$'s in $A$ $=$ $s$}\rbrace.
		\end{aligned}
	\end{equation*}
    For any $A\in S^{\gamma}$, 
    \begin{align*}
    	l = \text{the number of $\bf{c}$'s} - \text{the number of $\bf{b}$'s},
    \end{align*}
    since $J_{(0)}\bold{b}^i=-\bold{b}^i$, $J_{(0)}\bold{c}^i=\bold{c}^i$, and $J_{(0)}\bold{B}^i = J_{(0)} \boldsymbol{\Gamma}^i =0$.
	
	For any $A\in S^\gamma[k,l,s]$,  $k\geq$ the number of $\mathbf B'$s in $A$ plus the number of $\mathbf \Gamma'$s in $A$.  So when 
	$k<|s|$, $S^{\gamma}[k,l,s]=\emptyset$.

	Let $SW^{\gamma}[k, l]$, $SW^{\gamma}[k, l, s]$ be the subspaces of $SW^{\gamma}$ spanned by $S^{\gamma}[k,l]$, $S^{\gamma}[k,l,s]$, respectively. Accordingly, the subbundles $SW(\bar{T}^{*}X)[k,l]$ and $SW(\bar{T}^{*}X)[k,l,s]$ of $SW(\bar{T}^{*}X)$ have fibers isomorphic to $SW^{\gamma}[k,l]$, $SW^{\gamma}[k,l,s]$, respectively on any $x\in U$. 
	
	Let $H^0(X, SW(\bar{T}^{*}X))$, $H^0(X, SW(\bar{T}^{*}X))[k,l]$,  $H^0(X, SW(\bar{T}^{*}X))[k,l,s]$ be the space of holomorphic sections of $SW(\bar{T}^{*}X)$, $SW(\bar{T}^{*}X)[k,l]$, $SW(\bar{T}^{*}X)[k,l,s]$, respectively.
	
	\subsection{ Chern connection.}
	
	The K\"{a}hler form $h$ induces a Hermition metric on $SW(\bar{T}^{*}X)$. Let $\nabla=\nabla^{1,0}+\bar{\partial}^{\prime}$ be the Chern connection corresponding to the metric, and it satisfies $\nabla ab=(\nabla a)b+(-1)^{|a|}a\nabla b$ for any $a, b \in \Omega^{*, *}_{X}(X, SW(\bar{T}^{*}X))$.
	
	Let $\theta_{ij}=-\sum_{k}H^{ik}\bar{\partial}^{\prime}H_{kj}$ be a (0, 1)-form, then:
	\begin{equation}\label{eqn:hol}
		\nabla^{1,0}\bm{B}^i_{(k)}=\nabla^{1,0}\bm{\Gamma}^i_{(k)}=\nabla^{1,0}\bm{b}^i_{(k)}=\nabla^{1,0}\bm{c}^i_{(k)}=0
	\end{equation}
	and 
	\begin{equation}\label{eqn:antihol}
		\begin{aligned}
			\bar{\partial}^{\prime}\bm{B}^i_{(k)}&=\sum_{j}\theta_{ij}\bm{B}^j_{(k)}, \bar{\partial}^{\prime}\bm{b}^i_{(k)}=\sum_{j}\theta_{ij}\bm{b}^j_{(k)},  \\
			\bar{\partial}^{\prime}\bm{\Gamma}^i_{(k)}&=-\sum_{j}\theta_{ji}\bm{\Gamma}^j_{(k)},  
			\bar{\partial}^{\prime}\bm{c}^i_{(k)}=-\sum_{j}\theta_{ji}\bm{c}^j_{(k)}  	
		\end{aligned}
	\end{equation}
	for $k<0$. Using these expressions, for any $a \in S^{\gamma}$, the $\nabla a$ can be expressed as $\nabla a=\sum_{a'\in S^{\gamma}}\Theta_{a,a'}a'$, where $\Theta_{a,a'}$'s are one forms, and they are nonzero only if $a'$ has the same number of $\bm{B}$'s, $\bm{\Gamma}$'s, $\bm{b}$'s and $\bm{c}$'s as $a$. 
	
	\subsection{ A canonical isomorphism of sheaves.}
	
		Let $\bar{\mathcal{O}}(SW(\bar{T}^*X))$ be the sheaf of antiholomorphic sections of $SW(\bar{T}^*X)$.
		
		 Let $\Omega^{0,k}_X$ be the sheaf of smooth $(0,k)$-forms with values in $\mathbb{C}$ and let $\Omega^{0,*}_X=\bigoplus \Omega^{0,k}_X$. Let $ \Omega^{0,k}_X(SW(\bar{T}^*X))$ be the sheaf of smooth (0, k)-forms with values in $SW(\bar{T}^*X)$ and let $\Omega^{0,*}_X(SW(\bar{T}^*X))=\bigoplus \Omega^{0,k}_X(SW(\bar{T}^*X))$. 
		
     For any open subset $U$ of $X$, construct a vector space homomorphism $I_{U}$ locally:
    \begin{equation*}
    	I_U: \Omega^{0,*}_X(SW(\bar{T}^*X))(U)\cong SW^{\gamma}\otimes_{\mathbb{C}}\Omega^{0,*}_X(U)\rightarrow \Omega^{ch,*}_{X}(U), a\otimes f \mapsto af,
    \end{equation*}
    then by equation (\ref{eqn:transform}), a homomorphism $I$ of sheaves:
    \begin{equation*}
    	I: \Omega^{0,*}_X(SW(\bar{T}^*X))\rightarrow \Omega_X^{ch,*},
    \end{equation*} 
    is constructed.
			It is proved in \cite{S3}  that:
		\begin{theorem}\label{thm:iso}
			The homomorphism $I$ is an isomorphism of $\bar{\mathcal{O}}(SW(\bar{T}^*X))$ modules.
	\end{theorem}
	
		Let $\bar{D}=I^*(\bar{\partial})$, i.e.,
		$$\bar{D}a=I_U^{-1}(\bar{\partial}I_U(a)),$$ for any $a \in \Omega^{0,*}_X(SW(\bar{T}^*X))(U)$.
	According to the above Theorem \ref{thm:iso} and Theorem \ref{thm:iso0}, we get the following theorem, 
	\begin{theorem} \label{thm:iso2}
	\begin{equation*}
		H^i(X, \Omega^{ch}_X) \cong H^i(\Omega^{ch,*}_X, \bar{\partial})\cong H^i(\Omega^{0,*}_X(SW(\bar{T}^*X))(X), \bar{D}).
	\end{equation*}
	In particular:
	\begin{equation*}
		\Gamma(X, \Omega^{ch}_X)\cong H^0(\Omega^{ch,*}_X, \bar{\partial})\cong H^0(\Omega^{0,*}_X(SW(\bar{T}^*X))(X), \bar{D}).
	\end{equation*} 
\end{theorem}		
So we can use 	$H^0(\Omega^{0,*}_X(SW(\bar{T}^*X))(X), \bar{D})$ to calculate the space of global sections of chiral de Rham complex on $X$.

	By \cite{S3}, the operator $\bar{D}$ can be written in the following form:
	$$\bar D=\bar\partial'+\sum_{i=1}^\infty F_i.$$
	Here $F_i$, mapping  $\Omega^{0,*}_X(SW(\bar T^{*}X)[k,l,m])$ to $\Omega^{0,*+1}_X(SW(\bar T^{*}X)[k,l,m-i])$,   is a first order differential operator.
	Locally, for $a\in SW^\gamma$, $f\in \Omega^{0,*}_X(U)$,
	\begin{eqnarray}\label{eqn:Fn} I_U(F_n(a\otimes f))=\sum_{i, j}\sum_{s\in\mathcal Z, |s|=n}\frac{1}{s!} [::\varGamma^s\mathbf c^{i}:\mathbf b^j:_{(0)}, a]f\theta(s+e_i,j)
		\\
		+ \sum_j\sum_{s\in\mathcal Z, |s|=n+1} \frac{1}{s!}[:\varGamma^s\mathbf B^j:_{(0)}-\varGamma^s_{(-1)}\mathbf B^j_{(0)}, a]f\theta(s,j)+\sum_j
	\sum_{s\in\mathcal Z,|s|=n} \frac{1}{s!}[\varGamma^s_{(-1)}, a](\mathbf B_{(0)}^j f)\theta(s,j).\nonumber
	\end{eqnarray}
Here $\mathcal Z =\mathbb Z_{\geq 0}^{d}$ and $e_i=(0,\cdots,0,1,0\cdots,0)\in \mathcal Z$. For $s=(s_1,\cdots, s_d)\in \mathcal Z$, 
$$|s|=\sum_{i=1}^d s_i, \quad s!=\prod_{i=1}^d s_i!.$$
For any symbol `$X$', 
$$X^s=\prod_{i=1}^d (X^i)^{s_i}. $$
	 For $s\in \mathcal Z$,  $\theta(s+e_i,j)=\mathbf B_{(0)}^s\theta_{ij}$.
	 Let
	 $$\varGamma^i(z)=\sum_{j\neq 0} \frac{1}{-j}\mathbf \Gamma^i_{(j)} z^{-j},$$ which is a formal distribution, but is not a field in a vertex algebra for any $i$.
	 
\subsection{ Hermitian locally symmetric space}\label{sec3.4}
A Hermitian locally symmetric space is a connected complex manifold with a Hermitian structure such that its curvature (1,1)-form is flat.
Assume that $X$ is a Hermitian locally symmetric space and $R(\cdot,\cdot): TX\to TX$ is its curvature (1,1)-form. Locally
$$R(\cdot ,\cdot  )\frac{\partial}{\partial {\gamma^j}}=\bar \partial^{\prime}(\partial H_{jk}H^{ki})\frac{\partial}{\partial {\gamma^i}}=-\bar\partial^{\prime} \bar \theta_{ij}\frac{\partial}{\partial {\gamma^i}}=-\overline{\partial \theta_{ij}}\frac{\partial}{\partial {\gamma^i}}.$$
So 
\begin{equation}\label{eqn:curv}
	R(\cdot, \overline{H^{kl}\frac{\partial}{\partial {\gamma^l}}})\frac{\partial}{\partial {\gamma^j}}=
-\overline{\mathbf B^{k}_{(0)}\theta_{ij}}\frac{\partial}{\partial {\gamma^i}}.
\end{equation}
And $$\overline{\partial (\mathbf B^{k}_{(0)}\theta_{ij})}\frac{\partial}{\partial {\gamma^i}}=\bar\partial^{\prime} (\overline{\mathbf B^{k}_{(0)}\theta_{ij}}\frac{\partial}{\partial {\gamma^i}})=-\bar\partial^{\prime} (R(\cdot, \overline{H^{kl}\frac{\partial}{\partial {\gamma^l}}})\frac{\partial}{\partial {\gamma^j}})=-(\nabla_{\bar\partial^{\prime}} R)(\cdot, \overline{H^{kl}\frac{\partial}{\partial {\gamma^l}}})\frac{\partial}{\partial {\gamma^j}}=0.$$
Thus \begin{equation}\label{eqn:theta}
\partial (\mathbf B^{k}_{(0)}\theta_{ij})=0.
\end{equation}
 We have the following lemmas for Hermitian locally symmetric spaces,
\begin{lemma}\label{lem:symmetric}
		If $X$ is a Hermitian locally symmetric space, then, 
	$$F_n=0, \quad n\geq 3.$$
	So 
	\begin{equation*}
		\bar{D}=\bar{\partial}^{\prime}+F_1+F_2.
	\end{equation*}
	Locally, for $a\in SW^\gamma$, $f\in \Omega^{0,*}_X(U)$,
\begin{eqnarray}\label{eqn:F1}
	I_U(F_1(a\otimes f))=-\sum_{i,j,k} [::\varGamma^k\mathbf c^{i}:\mathbf b^j:_{(0)}, a]f	\overline{\langle R(\cdot, \overline{H^{kl}\frac{\partial}{\partial {\gamma^l}}})\frac{\partial}{\partial {\gamma^j}}, d \gamma^i\rangle}
	\\
	-\sum_{i,j,k} \frac{1}{2}[::\varGamma^k\varGamma^i:\mathbf B^j:_{(0)}-:\varGamma^k\varGamma^i:_{(-1)}\mathbf B^j_{(0)}, a]f	\overline{\langle R(\cdot, \overline{H^{kl}\frac{\partial}{\partial {\gamma^l}}})\frac{\partial}{\partial {\gamma^j}}, d \gamma^i\rangle};\nonumber
\\
\label{eqn:F2}
 I_U(F_2(a\otimes f))=-\sum_{i,j,k} \frac{1}{2}[:\varGamma^k\varGamma^i:_{(-1)}, a](\mathbf B_{(0)}^j f)	\overline{\langle R(\cdot, \overline{H^{kl}\frac{\partial}{\partial {\gamma^l}}})\frac{\partial}{\partial {\gamma^j}}, d \gamma^i\rangle}.
\end{eqnarray}
Here $R(\cdot,\cdot): TX\to TX$ is the curvature (1,1)-form, and $\langle \cdot, \cdot\rangle $ is the contraction of tangent field with $1$-form. 
\end{lemma}
\begin{proof}
If $X$ is a Hermitian locally symmetric space, then  $\theta(s, j)=0$ if $|s|\geq 3$  by equation (\ref{eqn:theta}), and $F_n=0$ by equation (\ref{eqn:Fn}).  From equation (\ref{eqn:Fn}) and the fact that $\varGamma^i_{(-1)}=0$,
\begin{eqnarray*}I_U(F_1(a\otimes f))&=&\sum_{i,j,k} [::\varGamma^k\mathbf c^{i}:\mathbf b^j:_{(0)}, a]f\mathbf B^k_{(0)}\theta_{ij}
	\\
	&&	+ \sum_{i,j,k} \frac{1}{2}[::\varGamma^k\varGamma^i:\mathbf B^j:_{(0)}-:\varGamma^k\varGamma^i:_{(-1)}\mathbf B^j_{(0)}, a]f\mathbf B^k_{(0)}\theta_{ij};
	\\
	I_U(F_2(a\otimes f))&=&\sum_{i,j,k} \frac{1}{2}[:\varGamma^k\varGamma^i:_{(-1)}, a](\mathbf B_{(0)}^j f)\mathbf B^k_{(0)}\theta_{ij}.
\end{eqnarray*}
	By equation (\ref{eqn:curv}),
$$\mathbf B^k_{(0)}\theta_{ij}=-\overline{\langle R(\cdot, \overline{H^{kl}\frac{\partial}{\partial {\gamma^l}}})\frac{\partial}{\partial {\gamma^j}}, d \gamma^i\rangle}.
$$
\end{proof}

So the operators $F_1$ and $F_2$ are determined by the curvature of $X$.

\begin{lemma}\label{lem:DeltaF1}
	If $X$ is a Hermitian locally symmetric space, then 
	\begin{equation*}
[\nabla, F_1]=0.
	\end{equation*}
	
\end{lemma}
\begin{proof}
 In the second summation of the right-hand side of the equation (\ref{eqn:F1}), the term containing $\mathbf B^j_{(0)}$ is subtracted.
Now $\nabla$ commutes with
$$\sum_k\varGamma^k_{(m)}\otimes H^{kl}\frac{\partial}{\partial {\gamma^l}}, \quad \sum_i \mathbf c^i_{(m)}\otimes d\bar \gamma^i,\quad \sum_i \mathbf \varGamma^i_{(m)}\otimes d\bar \gamma^i,\quad  \sum_j \mathbf b^j_{(m)}\otimes \frac{\partial}{\partial \bar \gamma^j}$$
for any $m\in \mathbb Z,$ and commutes with
$$\sum_j \mathbf B^j_{(n)}\otimes \frac{\partial}{\partial \bar \gamma^j}, \quad n\neq 0.$$

From the expression of $F_1$ in equation (\ref{eqn:F1}) and the fact that $\nabla R=0$, we can see that $[\nabla, F_1]=0$.
\end{proof}
For the convenience of deriving some useful formulas about $F_2$, we introduce an operator $\tilde F_2^j$ on $\Omega^{0,*}(SW(\bar T^*X))$ which is locally given by:
$$I_U(\tilde F_2^j(a\otimes f))=-\sum_{i,k} \frac{1}{2}[:\varGamma^k\varGamma^i:_{(-1)}, a]f 	\overline{\langle R(\cdot, \overline{H^{kl}\frac{\partial}{\partial {\gamma^l}}})\frac{\partial}{\partial {\gamma^j}}, d \gamma^i\rangle}.$$
Then by equation (\ref{eqn:F2}),
$$F_2=\sum_{j,l}\tilde F_2^j H^{jl}\nabla_{\frac{\partial}{\partial \gamma^l}}.
$$

\begin{lemma}\label{lem:DeltaF2}
	If $X$ is a Hermitian locally symmetric space, then 
	\begin{equation*}
		[\sum_{m}H^{im}\nabla_{\frac{\partial}{\partial \gamma^m}}, F_2] =0,\quad 
	[\nabla_{\frac{\partial}{\partial\bar \gamma^i}}, F_2] =\tilde F_2^j\tilde R(H^{jl}\frac{\partial}{\partial \gamma^l}, \frac{\partial}{\partial\bar \gamma^i}).
\end{equation*}
Here $\tilde R=[\bar \partial', \nabla^{1,0}]$ is the curvature of $SW(\bar T^*X)$.
\end{lemma}
\begin{proof}

$\nabla$ commutes with 
$$
\sum_k\varGamma^k_{(m)}\otimes H^{kl}\frac{\partial}{\partial {\gamma^l}}, \quad  \sum_i \mathbf \varGamma^i_{(m)}\otimes d\bar \gamma^i,\quad  \sum_j  H^{jl}\frac{\partial}{\partial \gamma^l}\otimes \frac{\partial}{\partial \bar \gamma^j}.
$$ 
By the definition of $\tilde F_2^j$ and the condition $\nabla R=0$, $\nabla$ commutes with $\tilde F_2^j\otimes  H^{jl}\frac{\partial}{\partial \gamma^l}.$ Since 
$$[\sum_{m}H^{im}\nabla_{\frac{\partial}{\partial \gamma^m}},\nabla^{1,0}]=0,$$ 
we have
$$	[\sum_{m}H^{im}\nabla_{\frac{\partial}{\partial \gamma^m}}, F_2]=	[\sum_{m}H^{im}\nabla_{\frac{\partial}{\partial \gamma^m}}, \sum_{j,l}\tilde F_2^j H^{jl}\nabla_{\frac{\partial}{\partial \gamma^l}}]=0.$$
$$	[\nabla_{\frac{\partial}{\partial\bar \gamma^i}}, F_2]=	[\nabla_{\frac{\partial}{\partial\bar \gamma^i}},  \sum_{j,l}\tilde F_2^j H^{jl}\nabla_{\frac{\partial}{\partial \gamma^l}}]=\tilde F_2^j\tilde R(H^{jl}\frac{\partial}{\partial \gamma^l}, \frac{\partial}{\partial\bar \gamma^i}).$$

\end{proof}

When  $X$ is a closed complex curve with genus $g \ge 2$, it can be constructed as a quotient $\mathbb{H}/\Gamma$, here $\mathbb{H}=\left\{z \in \mathbb{C}|\Ima(z)>0\right\}$ and $\Gamma$ is a discrete subgroup of the holomorphic automorphism group $\PSL(2, \mathbb{R})=\left\{z \mapsto \frac{az+b}{cz+d}| a, b, c, d \in \mathbb{R}, ad-bc=1\right\}$. Moreover, $\Gamma$ is required to act properly discontinuously without fixed points on $\mathbb{H}$ to make sure that $\mathbb{H}/\Gamma$ is compact.
We use the upper half plane $\mathbb{H}$ as a covering space $\tilde{X}$ of $X$, and a K\"{a}hler form $h$ on $X$ is locally given by: $h|_{U}=H d\gamma\wedge d\bar{\gamma}=\frac{i}{2 y^2}d\gamma\wedge d\bar{\gamma}$, where $y=\Ima \gamma>0$ and $\gamma$ is a holomorphic local coordinate. The generators of $ \mathcal W_+^\gamma$ are $\mathbf \Gamma,\mathbf c, \mathbf b,\mathbf B$. In this case $X$ is a Hermitian locally symmetric space,
$$\theta=-H^{-1}\bar\partial^{\prime} H=-\frac 2 {\gamma-\bar \gamma} d\bar\gamma, \quad \mathbf B_{(0)}\theta =i d\bar \gamma.$$
By Lemma \ref{lem:symmetric}, we have
\begin{corollary}\label{cor:D}
If $X$ is a complex curve with genus $g\geq 2$,
	$$\bar D=\bar \partial^{\prime}+F_1+F_2.$$
	The differential operators $F_1$ and $F_2$ are locally expressed as:
	\begin{equation}\label{eqn:F1curve}
		I_U(F_1(a\otimes f))=i[::\varGamma\bm{c}:\bm{b}:_{(0)}, a]fd\bar{\gamma}  
		+\frac{i}{2}[::\varGamma^{2}:\bm{B}:_{(0)}-\varGamma^2_{(-1)}\mathbf B_{(0)}, a]f d\bar{\gamma},
	\end{equation} 
	and
	\begin{equation}\label{eqn:F2curve}
		\begin{aligned}
			I_U(F_2(a\otimes f))&=[\varGamma^2_{(-1)}, a](\Ima \gamma)^2(\nabla_{\frac{\partial}{\partial \gamma}} f)d\bar{\gamma},
		\end{aligned}
	\end{equation}
    respectively.
	Here $a\in SW^{\gamma}$ and $f\in \Omega^{0,*}_{X}(U)$.
\end{corollary}	 	
\begin{lemma}\label{lem:F2curve}If $X$ is a complex curve with genus $g\geq 2$,
	\begin{equation*}
	iH^{-1}\nabla_{\frac{\partial}{\partial \gamma}}\nabla_{\frac{\partial}{\partial \bar{\gamma}}}F_2 =F_2iH^{-1}\nabla_{\frac{\partial}{\partial \bar{\gamma}}}\nabla_{\frac{\partial}{\partial \gamma}}.
\end{equation*}
\end{lemma}
\begin{proof}
	By Lemma \ref{lem:DeltaF2},
	\begin{eqnarray*}
	H^{-1}\nabla_{\frac{\partial}{\partial \gamma}}\nabla_{\frac{\partial}{\partial \bar{\gamma}}}F_2 &=&H^{-1}\nabla_{\frac{\partial}{\partial \gamma}}(F_2\nabla_{\frac{\partial}{\partial \bar{\gamma}}}+ \tilde F_2 \tilde R(H^{-1}\frac{\partial}{\partial \gamma}, \frac{\partial}{\partial\bar \gamma}))\\
	&=&F_2H^{-1}\nabla_{\frac{\partial}{\partial \gamma}}\nabla_{\frac{\partial}{\partial \bar{\gamma}}}+ \tilde F_2 H^{-1}\nabla_{\frac{\partial}{\partial \gamma}}\tilde R(H^{-1}\frac{\partial}{\partial \gamma}, \frac{\partial}{\partial\bar \gamma})   \\
	&=&F_2H^{-1}\nabla_{\frac{\partial}{\partial \gamma}}\nabla_{\frac{\partial}{\partial \bar{\gamma}}}+  F_2 \tilde R(H^{-1}\frac{\partial}{\partial \gamma}, \frac{\partial}{\partial\bar \gamma})  \\
		&=&F_2H^{-1}\nabla_{\frac{\partial}{\partial \bar{\gamma}}}\nabla_{\frac{\partial}{\partial \gamma}}.
	\end{eqnarray*}
Here $F_2=\tilde F_2 H^{-1}\nabla_{\frac{\partial}{\partial{\gamma}}}$ and $\tilde F_2$ commutes with $\nabla$ as in the proof of Lemma \ref{lem:DeltaF2}.	
\end{proof}

	\section{Global Sections on closed Complex Curves}\label{section4}
In this section we use Theorem \ref{thm:iso2} to calculate the space of global sections of chiral de Rham complex on closed complex curves with genus $g\geq 2$. Let $X$ be a closed complex curves with genus $g\geq 2$. As is mentioned in section \ref{sec3.4}, $X$ is uniformized by $\mathbb{H}$ and we choose the canonical metric induced by the K\"{a}hler form $h$ : $h|_{U}=H d\gamma\wedge d\bar{\gamma}=\frac{i}{2 y^2}d\gamma\wedge d\bar{\gamma}$, where $y=\Ima \gamma>0$ and $\gamma$ is a holomorphic local coordinate. The curvature of $X$ corresponding to this canonical metric is constant negative.

	\subsection{The curvature form on $SW(\bar{T}^{*}X)$.}
	
	For the curvature form $R \in \Omega^{1, 1}(\End(TX))$: $$R(\frac{\partial}{\partial \gamma}, \frac{\partial}{\partial \bar{\gamma}})=\nabla_{\frac{\partial}{\partial \bar{\gamma}}}\nabla_{\frac{\partial}{\partial \gamma}}-\nabla_{\frac{\partial}{\partial \gamma}}\nabla_{\frac{\partial}{\partial{\bar{\gamma}}}}.$$ Using the Hermitian metric on the upper half plane, a direct computation shows:
	\begin{equation*}
		-iH^{-1}R(\frac{\partial}{\partial \gamma}, \frac{\partial}{\partial \bar{\gamma}})\frac{\partial}{\partial \gamma}=-\frac{\partial}{\partial \gamma}.
	\end{equation*}
We have:
    \begin{equation}\label{eqn:Bb}
    -iH^{-1}\tilde{R}(\frac{\partial}{\partial \gamma}, \frac{\partial}{\partial \bar{\gamma}})A=-A,
  \quad \text{if } A=\bm{B}_{(k)}\text{ or } A=\bm{b}_{(k)},
      \end{equation}
      which follows from the fact that the antiholomorphic cotangent bundle $\bar{T}^{*}X$ of $X$ has curvature $-1.$
      
  And:
	\begin{equation}\label{eqn:cc}
		-iH^{-1}\tilde{R}(\frac{\partial}{\partial \gamma}, \frac{\partial}{\partial \bar{\gamma}})A=A,
\quad \text{ if }A=\bm{\Gamma}_{(k)}\text{ or } A=\bm{c}_{(k)},
	\end{equation}
	which follows from the fact that the antiholomorphic tangent bundle $\bar{T}X$ of $X$ has curvature $1.$
	
	 We obtain the following result:
	\begin{lemma}\label{lem:curture}
		The action of the curvature form $\tilde{R}$ on a section $a_{s}$ of $SW(\bar{T}^{*}(X))[k,l,s]$ is:
		\begin{equation}\label{eqn:c}
			-iH^{-1}\tilde{R}(\frac{\partial}{\partial \gamma}, \frac{\partial}{\partial \bar{\gamma}})a_s=(l-s)a_s.
		\end{equation}
		Therefore, the curvature of the sub bundle $SW(\bar{T}^{*}(X))[k,l,s]$ of $SW(\bar{T}^{*}(X))$ is just $l-s$.
		\begin{proof}
			We observe that:
			\begin{align*}
				l=\text{the number of $\bm{c}$'s}-\text{the number of $\bm{b}$'s},
			\end{align*} 
		and 
		\begin{align*}
			s=\text{the number of $\bm{B}$'s}-\text{the number of $\bm{\Gamma}$'s}.
		\end{align*} 
	 So the equation (\ref{eqn:c}) follows from the equations (\ref{eqn:Bb}--\ref{eqn:cc}) above and the fact that  $\tilde{R}_{E\otimes F}=\tilde{R}_{E}\otimes I_{F}+I_{E}\otimes \tilde{R}_{F}$ and $\tilde{R}_{E^{*}}=-\tilde{R}_{E}^*$ for any complex vector bundles $E$, $F$, here $E^{*}$ is the dual bundle of $E$. 
		\end{proof}
		
	\end{lemma}

	\subsection{The equation $\bar{D}a=0$.} \label{sec:main}
	 A nonzero section $a \in H^0(\Omega^{0,*}_X(SW(\bar{T}^*X)[k,l])(X), \bar{D})$, is a nonzero smooth section $a$  of $SW(\bar{T}^*X)[k,l]$ with $\bar D a=0$. The section $a$ can be written as  $$a=\sum_{i\geq 0} a_{s-i},$$ where $a_{s-i}$ is a smooth section of $SW(\bar{T}^{*}X)[k, l, s-i]$  and $a_s$ is nonzero. The summation is finite since $SW(\bar{T}^{*}X)[k, l, s-i]=\mathbf 0$ when $k<i-s$.
  By Corollary \ref{cor:D} , $\bar D a=0$ is  divided into a series of equations bellow: 
\begin{eqnarray*}
		SW(\bar{T}^{*}X)[k,l,s]: &&\bar{\partial}^{\prime}a_s=0;\\
	SW(\bar{T}^{*}X)[k,l,s-1]: &&\bar{\partial}^{\prime}a_{s-1}+F_1a_s=0;\\
	SW(\bar{T}^{*}X)[k,l,s-2]: &&\bar{\partial}^{\prime}a_{s-2}+F_1a_{s-1}+F_2a_{s}=0;\\
	&&	\ldots\\
SW(\bar{T}^{*}X)[k,l,s-i]: &&\bar{\partial}^{\prime}a_{s-i}+F_1a_{s-i+1}+F_2a_{s-i+2}=0;\\
&&	\ldots
\end{eqnarray*}
Therefore, $a_s$ must be a nonzero holomorphic section of $SW(\bar{T}^{*}X)[k,l,s]$.
\subsubsection{Case 1: $l-s<0$}	
	The following vanishing theorem for holomorphic sections of Hermitian vector bundles is well known (cf. [ \cite{Ko}, III, 1.9(ii)]),
	\begin{theorem}\label{thm:vanish}
		Let E be a Hermitian vector bundle over a compact Hermition manifold M. Let R be the curvature of E, and $\hat{K}$ be the mean curvature of E. If $\hat{K}$ is negative-definite everywhere on M, then E admits no nonzero holomorphic sections.
	\end{theorem}
When $l-s < 0$, by Lemma \ref{lem:curture}, the mean curvature of $SW(\bar{T}^{*}X)[k,l,s]$ is negative-definite everywhere on $X$, so by above Theorem \ref{thm:vanish}, there is no nonzero holomorphic section $a_s$.
	

\subsubsection{ Case 2: $l-s > 0$.}
	
Let $a_s$ be a holomorphic section of $SW(\bar{T}^{*}X)[k, l, s]$ with $l-s>0$. 
Let  $\bar{\partial}^{\prime*}$ be the dual operator of $\bar{\partial}^{\prime}$, then 
$\bar{\partial}^{\prime*}=-iH^{-1}\iota_{\frac{\partial}{\partial \bar{\gamma}}} \nabla_{\frac{\partial}{\partial \gamma}}$. 
	\begin{equation}\label{eqn: compo}
	\begin{aligned}
	\bar{\partial}^{\prime}\bar{\partial}^{\prime*} &=-d\bar{\gamma}\nabla_{\frac{\partial}{\partial\bar{\gamma}}}iH^{-1}\iota_{\frac{\partial}{\partial \bar{\gamma}}} \nabla_{\frac{\partial}{\partial \gamma}}     \\
	&=-id\bar{\gamma}H^{-1}\iota_{\frac{\partial}{\partial \bar{\gamma}}}\nabla_{\frac{\partial}{\partial \bar{\gamma}}}\nabla_{\frac{\partial}{\partial \gamma}}  \\
	&=-id\bar{\gamma}H^{-1}\iota_{\frac{\partial}{\partial \bar{\gamma}}}(\nabla_{\frac{\partial}{\partial \gamma}}\nabla_{\frac{\partial}{\partial \bar{\gamma}}}+\tilde{R}(\frac{\partial}{\partial \gamma}, \frac{\partial}{\partial \bar{\gamma}})).
	\end{aligned}
    \end{equation}

	\begin{lemma}     \label{lem:first}
		Given a nonzero holomorphic section $a_s$ of $SW(\bar{T}^{*}X) [k, l, s]$ with $l-s>0$,  $a_{s-1}=\frac{-1}{l-s}\bar{\partial}^{\prime*}F_1 a_s$ is a solution of the equation $\bar{\partial}^{\prime}a_{s-1}+F_1a_s=0$.
		
		\begin{proof}
			We have
			\begin{align*}
				\bar{\partial}^{\prime}\bar{\partial}^{\prime*}F_1a_s 
				&=-id\bar{\gamma}H^{-1}\iota_{\frac{\partial}{\partial \bar{\gamma}}}(\nabla_{\frac{\partial}{\partial \gamma}}\nabla_{\frac{\partial}{\partial \bar{\gamma}}}+\tilde{R}(\frac{\partial}{\partial \gamma}, \frac{\partial}{\partial \bar{\gamma}}))F_1 a_s   \\
				&=-id\bar{\gamma}H^{-1}\iota_{\frac{\partial}{\partial \bar{\gamma}}}\nabla_{\frac{\partial}{\partial \gamma}}\nabla_{\frac{\partial}{\partial \bar{\gamma}}}F_1a_s+(l-s)F_1a_s  \\
				&= (l-s)F_1a_s.
			\end{align*}
		  The last equality follows from  Lemma \ref{lem:DeltaF1} and the fact that $a_s$ is a holomorphic section. So $a_{s-1}=\frac{-1}{l-s}\bar{\partial}^{\prime*}F_1 a_s$ is a solution of the equation. 
		\end{proof}
		
	\end{lemma}
	
	\begin{lemma}\label{lem:second}
		Suppose $a_{s-1}$ is given as in Lemma \ref{lem:first}, then $a_{s-2}=\frac{-1}{2l-2s+1}(\bar{\partial}^{\prime*}F_1a_{s-1}+\bar{\partial}^{\prime*}F_2a_s)$ is a solution of the equation $\bar{\partial}^{\prime}a_{s-2}+F_1a_{s-1}+F_2a_{s}=0$.
		
		\begin{proof}
			Since $\bar{\partial}^{\prime}a_{s-1}+F_1a_s=0$, we have:
			\begin{align*}
				\nabla_{\frac{\partial}{\partial\bar{\gamma}}}a_{s-1}=-\iota_{\partial\bar{\gamma}}F_1a_s.
			\end{align*}
			We compute:
			\begin{align*}
				\bar{\partial}^{\prime}\bar{\partial}^{\prime*}F_1a_{s-1} &=-d\bar{\gamma} iH^{-1} \iota_{\frac{\partial}{\partial\bar{\gamma}}}(\nabla_{\frac{\partial}{\partial \gamma}}\nabla_{\frac{\partial}{\partial\bar{\gamma}}}+\tilde{R}(\frac{\partial}{\partial \gamma}, \frac{\partial}{\partial \bar{\gamma}}))F_1a_{s-1}   \\
				&=-d\bar{\gamma} iH^{-1} \iota_{\frac{\partial}{\partial\bar{\gamma}}}\nabla_{\frac{\partial}{\partial \gamma}}F_1\nabla_{\frac{\partial}{\partial\bar{\gamma}}}a_{s-1}+(l-s+1)F_1a_{s-1} \\
				&=d\bar{\gamma} iH^{-1} \iota_{\frac{\partial}{\partial\bar{\gamma}}}\nabla_{\frac{\partial}{\partial \gamma}}F_1\iota_{\frac{\partial}{\partial \bar{\gamma}}}F_1a_s+(l-s+1)F_1a_{s-1}.  
			\end{align*}
			
			Using the expression in Lemma \ref{lem:first}, we obtain:
			\begin{align*}
				F_1a_{s-1} &= d\bar{\gamma} \iota_{\frac{\partial}{\partial \bar{\gamma}}} F_1 \frac{-1}{l-s} \bar{\partial}^{\prime*}F_1a_s   \\
				&=\frac{1}{l-s} d \bar{\gamma} \iota_{\frac{\partial}{\partial \bar{\gamma}}} F_1 i H^{-1} \iota_{\frac{\partial}{\partial \bar{\gamma}}} \nabla_{\frac{\partial}{\partial \gamma}} F_1a_s  \\
				&=\frac{1}{l-s} i H^{-1} d\bar{\gamma} \nabla_{\frac{\partial}{\partial \gamma}} \iota_{\frac{\partial}{\partial \bar{\gamma}}} F_1 \iota_{\frac{\partial}{\partial \bar{\gamma}}} F_1 a_s.
			\end{align*}
			
			Therefore, 
			\begin{align*}
				\bar{\partial}^{\prime}\bar{\partial}^{\prime*}F_1a_{s-1}=(2l-2s+1)F_1a_{s-1}.
			\end{align*}
			
			For the term $\bar{\partial}^{\prime}\bar{\partial}^{\prime*}F_2a_s$, using Lemma \ref{lem:F2curve}, we have:
			\begin{align*}
				\bar{\partial}^{\prime}\bar{\partial}^{\prime*}F_2a_s &=-d\bar{\gamma} iH^{-1} \iota_{\frac{\partial}{\partial\bar{\gamma}}}(\nabla_{\frac{\partial}{\partial \gamma}}\nabla_{\frac{\partial}{\partial\bar{\gamma}}}+\tilde{R}(\frac{\partial}{\partial \gamma},\frac{\partial}{\partial \bar{\gamma}}))F_2a_{s}   \\
				&=-d\bar{\gamma} iH^{-1} \iota_{\frac{\partial}{\partial\bar{\gamma}}}\nabla_{\frac{\partial}{\partial \gamma}}\nabla_{\frac{\partial}{\partial\bar{\gamma}}}F_2a_s+(l-s+1)F_2a_s   \\
				&=-F_2iH^{-1}\nabla_{\frac{\partial}{\partial \bar{\gamma}}}\nabla_{\frac{\partial}{\partial \gamma}}a_s+(l-s+1)F_2a_s.
			\end{align*}  
			
			Moreover,
			\begin{align*}
				iH^{-1} \nabla_{\frac{\partial}{\partial \bar{\gamma}}}\nabla_{\frac{\partial}{\partial \gamma}}a_s 
				&= iH^{-1} \tilde{R}(\frac{\partial}{\partial \gamma}, \frac{\partial}{\partial \bar{\gamma}})a_s+	iH^{-1} \nabla_{\frac{\partial}{\partial \gamma}}\nabla_{\frac{\partial}{\partial \bar{\gamma}}}a_s \\
				&= iH^{-1} \tilde{R}(\frac{\partial}{\partial \gamma}, \frac{\partial}{\partial \bar{\gamma}})a_s \\
				&= -(l-s)a_s.
			\end{align*}
			
			So:
			\begin{align*}
				\bar{\partial}^{\prime}\bar{\partial}^{\prime*}F_2a_s&=(l-s)F_2a_s+(l-s+1)F_2a_s  \\
				&=(2l-2s+1)F_2a_s.
			\end{align*}  
		
		Thus $a_{s-2}=\frac{-1}{2l-2s+1}(\bar{\partial}^{\prime*}F_1a_{s-1}+\bar{\partial}^{\prime*}F_2a_s)$ is a solution of the equation.
		\end{proof}
	\end{lemma}
	
	In general:
	
	\begin{theorem}\label{thm:solution}
		Let $a_{s+1}=0$,
		then the formula 
		\begin{equation}\label{eqn:induct}
			a_{s-t}=\frac{-1}{t(l-s)+\frac{1}{2}t(t-1)}(\bar{\partial}^{\prime*}F_1a_{s-t+1}+\bar{\partial}^{\prime*}F_2a_{s-t+2})
		\end{equation}
		for $t\geq1$ gives a solution of the equation $\bar{D}(\sum_{t \geq 0} a_{s-t})=0$.

		\begin{proof}
			Induction on $t$. The cases $t=1, 2$ are verified in Lemma \ref{lem:first} and Lemma \ref{lem:second}, respectively. Next assume that $t\geq3$.
			
			 By the inductive hypothesis $\bar{\partial}^{\prime}a_{s-t+1}+F_1a_{s-t+2}+F_2a_{s-t+3}=0$, we have:
			\begin{align*}
				\nabla_{\frac{\partial}{\partial\bar{\gamma}}}a_{s-t+1}+\iota_{\frac{\partial}{\partial \bar{\gamma}}}F_1a_{s-t+2}+\iota_{\frac{\partial}{\partial \bar{\gamma}}}F_2a_{s-t+3}=0.
			\end{align*}
			We compute:
			\begin{align*}
				\bar{\partial}^{\prime}\bar{\partial}^{\prime*}F_1a_{s-t+1}&=-d\bar{\gamma} iH^{-1} \iota_{\frac{\partial}{\partial\bar{\gamma}}}(\nabla_{\frac{\partial}{\partial \gamma}}\nabla_{\frac{\partial}{\partial\bar{\gamma}}}+\tilde{R}(\frac{\partial}{\partial \gamma},\frac{\partial}{\partial\bar{\gamma}}))F_1a_{s-t+1}  \\
				&=-iH^{-1} \nabla_{\frac{\partial}{\partial \gamma}}\nabla_{\frac{\partial}{\partial\bar{\gamma}}}F_1a_{s-t+1}+(l-s+t-1)F_1a_{s-t+1}  \\
				&=-iH^{-1} \nabla_{\frac{\partial}{\partial \gamma}}F_1\nabla_{\frac{\partial}{\partial\bar{\gamma}}}a_{s-t+1}+(l-s+t-1)F_1a_{s-t+1}  \\
				&=iH^{-1} \nabla_{\frac{\partial}{\partial \gamma}}F_1\iota_{\frac{\partial}{\partial \bar{\gamma}}}(F_1a_{s-t+2}+F_2 a_{s-t+3})+(l-s+t-1)F_1a_{s-t+1}.
			\end{align*}
			
			By the inductive formula of $a_{s-t+1}$,
			\begin{align*}
				F_1a_{s-t+1} &=F_1\frac{-1}{(t-1)(l-s)+\frac{1}{2}(t-1)(t-2)}(\bar{\partial}^{\prime*}F_1a_{s-t+2}+\bar{\partial}^{\prime*}F_2a_{s-t+3})   \\
				&=F_1\frac{1}{(t-1)(l-s)+\frac{1}{2}(t-1)(t-2)}iH^{-1}\iota_{\frac{\partial}{\partial \bar{\gamma}}}\nabla_{\frac{\partial}{\partial \gamma}}(F_1a_{s-t+2}+F_2a_{s-t+3})   \\
				&=\frac{1}{(t-1)(l-s)+\frac{1}{2}(t-1)(t-2)}iH^{-1} \nabla_{\frac{\partial}{\partial \gamma}}F_1\iota_{\frac{\partial}{\partial \bar{\gamma}}}(F_1a_{s-t+2}+F_2 a_{s-t+3}).
			\end{align*}
			
			Therefore,
			\begin{align*}
				\bar{\partial}^{\prime}\bar{\partial}^{\prime*}F_1a_{s-t+1} &=((t-1)(l-s)+\frac{1}{2}(t-1)(t-2))F_1a_{s-t+1}+(l-s+t-1)F_1a_{s-t+1}   \\
				&=(t(l-s)+\frac{1}{2}t(t-1))F_1a_{s-t+1}.
			\end{align*}
			
			Next we compute:
			\begin{align*}
				\bar{\partial}^{\prime}\bar{\partial}^{\prime*}F_2a_{s-t+2}  &=-d\bar{\gamma} iH^{-1} \iota_{\frac{\partial}{\partial\bar{\gamma}}}(\nabla_{\frac{\partial}{\partial \gamma}}\nabla_{\frac{\partial}{\partial\bar{\gamma}}}+\tilde{R}(\frac{\partial}{\partial \gamma},\frac{\partial}{\partial\bar{\gamma}}))F_2a_{s-t+2}  \\
				&=-iH^{-1} \nabla_{\frac{\partial}{\partial \gamma}}\nabla_{\frac{\partial}{\partial\bar{\gamma}}}F_2a_{s-t+2}+(l-s+t-1)F_2a_{s-t+2}   \\
				&=-F_2iH^{-1} \nabla_{\frac{\partial}{\partial \bar{\gamma}}}\nabla_{\frac{\partial}{\partial\gamma}}a_{s-t+2}+(l-s+t-1)F_2a_{s-t+2}.
			\end{align*}
			
			Moreover, by the inductive hypothesis $\bar{\partial}^{\prime}a_{s-t+2}+F_1a_{s-t+3}+F_2a_{s-t+4}=0$, we have:
			\begin{align*}
				\nabla_{\frac{\partial}{\partial \bar{\gamma}}}a_{s-t+2}+\iota_{\frac{\partial}{\partial \bar{\gamma}}}F_1a_{s-t+3}+\iota_{\frac{\partial}{\partial \bar{\gamma}}}F_2a_{s-t+4}=0.
			\end{align*}
		
		So:
			\begin{align*}
			&~~~-iH^{-1} \nabla_{\frac{\partial}{\partial \bar{\gamma}}}\nabla_{\frac{\partial}{\partial\gamma}}a_{s-t+2} \\
				&=-iH^{-1}\tilde{R}(\frac{\partial}{\partial \gamma}, \frac{\partial}{\partial\bar{\gamma}})a_{s-t+2}-iH^{-1}\nabla_{\frac{\partial}{\partial \gamma}}\nabla_{\frac{\partial}{\partial \bar{\gamma}}}a_{s-t+2}  \\
				&=(l-s+t-2)a_{s-t+2}+iH^{-1}\nabla_{\frac{\partial}{\partial \gamma}}\iota_{\frac{\partial}{\partial \bar{\gamma}}}F_1a_{s-t+3}+iH^{-1}\nabla_{\frac{\partial}{\partial \gamma}}\iota_{\frac{\partial}{\partial \bar{\gamma}}}F_2a_{s-t+4}.
			\end{align*}
			
			By the inductive formula of $a_{s-t+2}$,
			\begin{align*}
				F_2a_{s-t+2}&=F_2\frac{-1}{(t-2)(l-s)+\frac{1}{2}(t-2)(t-3)}(\bar{\partial}^{\prime*}F_1a_{s-t+3}+\bar{\partial}^{\prime*}F_2a_{s-t+4}) \\
				&=\frac{1}{(t-2)(l-s)+\frac{1}{2}(t-2)(t-3)}F_2iH^{-1}\iota_{\frac{\partial}{\partial \bar{\gamma}}}\nabla_{\frac{\partial}{\partial \gamma}}(F_1a_{s-t+3}+F_2a_{s-t+4}).
			\end{align*}
			
			Therefore,
			\begin{align*}
				&~~~\bar{\partial}^{\prime}\bar{\partial}^{\prime*}F_2a_{s-t+2}  \\
				&=((l-s+t-2)+((t-2)(l-s)+\frac{1}{2}(t-2)(t-3))+(l-s+t-1))F_2a_{s-t+2}  \\
				&=(t(l-s)+\frac{1}{2}t(t-1))F_2a_{s-t+2}.   
			\end{align*}
			
			So under the inductive hypothesis, $a_{s-t}=\frac{-1}{t(l-s)+\frac{1}{2}t(t-1)}(\bar{\partial}^{\prime*}F_1a_{s-t+1}+\bar{\partial}^{\prime*}F_2a_{s-t+2})$ is a solution of the equation $\bar{\partial}^{\prime}a_{s-t}+F_1 a_{s-t+1}+F_2a_{s-t+2}=0$.
		\end{proof}
	\end{theorem} 

  So for any $a_s\in H^{0}(X, SW(\bar{T}^{*}X))[k, l, s]$ with $l-s>0$, we can give an $a=\sum_{i\geq 0}{a_{s-i}}$ satisfying the condition for global sections.
	
\subsubsection{ Case 3: $l-s = 0$.}
	
	In this case, we assume that $a_s\in H^{0}(X, SW(\bar{T}^{*}X))[k, s, s]$.
	
	\begin{lemma}\label{lem:G} $SW(\bar{T}^{*}X)[k,s,s]$ is a trivial vector bundle. So
	every holomorphic section $a_s$ of $SW(\bar{T}^{*}X)[k,s,s]$  satisfies $\nabla a=0$ and  the space of holomorphic sections of $SW(\bar{T}^{*}X)[k,s,s]$ is isomorphic to its fiber $SW^{\gamma}[k,s,s]$.
\end{lemma} 
	
	\begin{proof}
	By the construction of 	$SW(\bar{T}^{*}X)[k,l,s]$, it is a direct  sum of tensors of $\bar{T}^{*}X$ and $\bar{T}X$ when $X$ is a curve. When $l=s$, the number of $\bar{T}^{*}X$ and $\bar{T}X$ are equal in the tensor. So $SW(\bar{T}^{*}X)[k,s,s]$ is a trivial vector bundle.

	\end{proof}

\begin{corollary}\label{cor:null}
	For any holomorphic section $a_s$ of $SW(\bar{T}^{*}X)[k, s, s]$, we have $F_2\,a_s=0$.
	\begin{proof}
		This follows directly from equation (\ref{eqn:F2curve}) and Lemma \ref{lem:G}.
	\end{proof}
\end{corollary}
	
	The Hodge decomposition of forms will be used later in the calculation of global sections:
	\begin{theorem}\label{thm:Hodgedecompostion}
		Let $X$ be a compact Hermitian manifold, and $E$ an Hermitian vector bundle over $X$, the operator $\Delta$ is given by $\Delta=\bar{\partial}^{\prime}\bar{\partial}^{\prime*}+\bar{\partial}^{\prime*}\bar{\partial}^{\prime}$,  $\Omega_{X}^{p,q}(E)$ is the space of $(p, q)$-forms over $X$ taking values in $E$, then: 
		\begin{equation}\label{eqn:hodge}
		\Omega_{X}^{p, q}(E)= \Ker\Delta|_{\Omega_{X}^{p, q}(E)} \oplus \bar{\partial}^{\prime} \Omega_{X}^{p, q-1}(E) \oplus \bar{\partial}^{\prime*} \Omega_{X}^{p, q+1}(E).
		\end{equation}
	\end{theorem}
	
	\begin{lemma}\label{lem:solution1}
		There exists a section $a_{s-1}$ of $SW(\bar{T}^{*}X)[k, s , s-1]$ such that $\bar{\partial}^{\prime}a_{s-1}+F_1a_s=0$ if and only if $F_1a_s=0$.
		\begin{proof}
			As in the proof of Lemma \ref{lem:first}, we have:
			\begin{align*}
				\bar{\partial}^{\prime}\bar{\partial}^{\prime*}F_1a_s=(s-s)F_1a_s=0.
			\end{align*}
			On the other hand:
			\begin{align*}
				\bar{\partial}^{\prime*}\bar{\partial}^{\prime}F_1a_s=0.
			\end{align*}
			Therefore,
			\begin{align*}
				\Delta F_1a_s=0.
			\end{align*}
		So $\bar{\partial}^{\prime}a_{s-1}+F_1a_s=0$, if and only if 
			\begin{align*}
				F_1a_s \in \Ker\Delta \cap \im \bar{\partial}^{\prime}.
			\end{align*}
	That is $F_1a_s=0$ 	by Theorem \ref{thm:Hodgedecompostion}.
		\end{proof}
	\end{lemma}

	\begin{theorem}\label{thm:solution2}
	  For a holomorphic section $a_s$  of $ SW(\bar{T}^{*}X)[k, s, s]$,  there is an $a=\sum_{i\geq 0}a_{s-i}$ with $\bar D a=0$ if and only if $F_1 a_s=0$. If $F_1 a_s=0$, then $\bar D a_s=0$.
	 \end{theorem}
 \begin{proof}
 	By Lemma \ref{lem:solution1}, if there is an $a=\sum_{i\geq 0}a_{s-i}$ with $\bar D a=0$, we must have $F_1a_s=0$. By Corollary \ref{cor:D} and Corollary \ref{cor:null}, if $F_1a_s=0$, then $\bar Da_s=0$.
 \end{proof}
 So  a holomorphic section $a_s$  of $ SW(\bar{T}^{*}X)[k, s, s]$  corresponds to a global section of $\Omega^{ch}_X$ if and only if $F_1a_s=0$ and in this case $a_s$ itself is a global section. 
 
 \subsection{ The $\mathfrak{sl}_2$ action on a $\beta\gamma-bc$ system}\label{sec4.3}
	
	In this subsection, we discuss the equation $F_1a_s=0$ when $a_s$ is a holomorphic section of $ SW(\bar{T}^{*}X)[k, s, s]$.

		Let $V$ be a complex linear space of dimension one. Note that when the complex manifold $X$ is of dimension one, $\mathcal{W}^{\gamma}_{+}$ is isomorphic to $\mathcal{W}_{+}(V)$ as a vertex algebra by mapping $\bm{B}$, $\bm{\Gamma}$, $\bm{b}$, $\bm{c}$ to $\beta$, $\alpha$, $b$, $c$, respectively. Let $$\tilde{\gamma}(z)=\gamma(z)-\gamma_{(-1)}=\sum_{n\neq -1}\gamma_{(n)}z^{-n-1},$$ which is a formal distribution, but is not a field in a vertex algebra. We naturally correspond
	$\bm{B}$, $\varGamma$, $\bm{b}$, $\bm{c}$ to $\beta$, $\tilde{\gamma}$, $b$, $c$, respectively.

	Let $T$ be the Lie algebra generated by vector fields  $\lbrace x^i\frac{\partial}{\partial x}\rbrace_{i=0, 1, 2}$ on $V$.
 $T$ is isomorphic to  $\mathfrak{sl}_2$ by the following relations:
	\begin{equation*}
		[x^i\frac{\partial}{\partial x}, x^j\frac{\partial}{\partial x}]=(j-i)x^{i+j-1}\frac{\partial}{\partial x}.
	\end{equation*} 
	 Let  $\mathcal{L}: T \rightarrow End_{\mathbb{C}}(\mathcal{W}(V))$  be the Lie algebra homomorphism (See Part III in \cite{MS2}) given by 
	\begin{equation*}
		\mathcal{L}(x^i\frac{\partial}{\partial x})=(Q_{(0)}:\gamma^i b:)_{(0)},
	\end{equation*}
	where $Q=:\beta c:$. $\mathcal{L}$ is naturally an action of $\mathfrak{sl}_2$ on the $\beta\gamma-bc$ system $\mathcal{W}(V)$, and the coefficients $(Q_{(0)}:\gamma^i b:)_{(n)}$ of the fields $Q_{(0)}:\gamma^ib:$ generate an affine Kac-Moody algebra $\hat{\mathfrak{sl}}_2$ at level $0$ inside $\mathcal{W}(V)$.

	Define a linear operator $\mathcal{L}^{+}: T \rightarrow End_{\mathbb{C}}(\mathcal{W}(V))$ as zero modes of formal distributions:
	\begin{equation*}
		\mathcal{L}^{+}(x^i\frac{\partial}{\partial x})=(:\tilde{\gamma}^i\beta:)_{(0)}-i(::\tilde{\gamma}^{i-1}b:c:)_{(0)},
	\end{equation*}
	which is obtained by expanding $\mathcal{L}(x^i\frac{\partial}{\partial x})$ as infinite sums of products of modes of $\beta$, $\gamma$, $b$ and $c$, and throwing out all terms where the mode $\gamma_{(-1)}$ appears. 
	Note that $\mathcal{L}^{+}$ is not a homomorphism of Lie algebras. The relation between $\mathcal{L}$ and	$\mathcal{L}^{+}$ is as follows:
	\begin{lemma}\label{lem:relationL}
		For the map $\mathcal{L}$ and the map $\mathcal{L}^{+}$, the following relations hold:
		\begin{eqnarray*}
			\mathcal{L}(\frac{\partial}{\partial x})&=&\mathcal{L}^{+}(\frac{\partial}{\partial x}),
	\\
			\mathcal{L}(x\frac{\partial}{\partial x})&=&\mathcal{L}^{+}(x\frac{\partial}{\partial x})+\gamma_{(-1)}\mathcal{L}^{+}(\frac{\partial}{\partial x}),
	\\
			\mathcal{L}(x^2\frac{\partial}{\partial x})&=&\mathcal{L}^{+}(x^2\frac{\partial}{\partial x})+\gamma_{(-1)}\mathcal{L}^{+}(2x\frac{\partial}{\partial x})+\gamma_{(-1)}^2\mathcal{L}^{+}(\frac{\partial}{\partial x}).
		\end{eqnarray*}
		\begin{proof}
			See Lemma 6 in \cite{S3}.
		\end{proof}
	\end{lemma}
	
	Let  $\mathcal{W}^{T}(V)$ be the space of invariants in $\mathcal W(V)$ under the action of $\mathfrak{sl}_{2}$:
	
	\begin{definition}
		The subspace $\mathcal{W}^{T}(V)$ of $\mathcal{W}(V)$ is defined by $$\mathcal{W}^{T}(V)=\left\{A\in\mathcal{W}(V)|\mathcal{L}(x^i\frac{\partial}{\partial x})A=0, i=0, 1, 2\right\}.$$
	\end{definition}
	
	\begin{proposition}
	 The subspace $\mathcal{W}^{T}(V)$ is a sub vertex algebra of $\mathcal{W}(V)$.
	\end{proposition}
	\begin{proof}
		For any $A$, $B$ $\in$ $\mathcal{W}(V)$, we have $$\mathcal{L}(x^i\frac{\partial}{\partial x})(A_{(n)}B)=(\mathcal{L}(x^i\frac{\partial}{\partial x})A)_{(n)}B+A_{(n)}(\mathcal{L}(x^i\frac{\partial}{\partial x})B).$$ Therefore, $\left\{A\in\mathcal{W}(V)|\mathcal{L}(x^i\frac{\partial}{\partial x})A=0\right\}$ are sub vertex algebras of $\mathcal{W}(V)$ and $\mathcal{W}^{T}(V)=\cap_{i=0, 1, 2}\left\{A\in\mathcal{W}(V)|\mathcal{L}(x^i\frac{\partial}{\partial x})A=0\right\}$ is a sub vertex algebra of $\mathcal{W}(V)$.
	\end{proof}
	
		Let \begin{equation}\label{M1}
			M_1:=\bigoplus_{k, s}(H^0(X, SW(\bar{T}^{*}X))[k, s, s]\cap \Ker F_1),
			\end{equation}
			where $F_1$ is a first order differential operator.
		
	\begin{proposition}
		The subspace $M_1$ is a sub vertex algebra of $H^0(X, \Omega^{ch}_X)$.
	\end{proposition}
	\begin{proof}
		By Lemma \ref{lem:G}, one only need to check that the subspace $\bigoplus_{k, s}(SW^{\gamma}[k, s, s]\cap \Ker F_1)$ of the fiber $SW^{\gamma}$ is a vertex algebra, which is obvious.
	\end{proof}
	
	\begin{theorem}\label{thm:inv}
		The subspace $M_1$
		 is isomorphic to $\mathcal{W}^{T}(V)$ as vertex algebras.
		\begin{proof}
			By Lemma \ref{lem:G},  the space of holomorphic sections of $\bigoplus_{k,s}SW(\bar{T}^{*}X)[k,s,s]$ is isomorphic to its fibre $\bigoplus_{k,s} SW^{\gamma}[k,s,s]$. That is the subspace of $SW^{\gamma}$ generated by monomials with
			the number of $\mathbf B$'s plus the number of $\mathbf b$'s is equal to the number of $\mathbf \Gamma$'s plus the number of $\mathbf c$'s.
			By Lemma \ref{lem:relationL}, 
			$$\mathcal W^T(V)=\{A\in \mathcal{W}(V)|\mathcal{L}^{+}(x^i\frac{\partial}{\partial x})A=0, i=0, 1, 2\}.$$
Now
$$\mathcal{L}^{+}(\frac{\partial}{\partial x})A=\beta_{(0)}A=0,$$  if and only if $A\in \mathcal{W}_{+}(V)$.
$$\mathcal{L}^{+}(x\frac{\partial}{\partial x})A=:\tilde{\gamma}\beta:_{(0)}A-:bc:_{(0)}A=0,$$ if and only if the number of $\beta$'s plus the number of $b$'s is equal to the number of $\alpha$'s plus the number of $c$'s in each monomial of $A$.
$$\mathcal{L}^{+}(x^2\frac{\partial}{\partial x})A=:\tilde{\gamma}\tilde{\gamma}\beta:_{(0)}A-2:\tilde{\gamma}bc:_{(0)}A=0,$$  is equivalent to $F_1a_s=0$ by the isomorphism between $\mathcal{W}_{+}^{\gamma}$ and $\mathcal{W}_{+}(V)$.

Since the isomorphism above between $M_1$ and $\mathcal{W}^{T}(V)$ is compatible with the vertex isomorphism between $\mathcal{W}_{+}^{\gamma}$ and $\mathcal{W}_{+}(V)$ by mapping $\bm{B}$, $\bm{\Gamma}$, $\bm{b}$, $\bm{c}$ to $\beta$, $\alpha$, $b$, $c$, respectively, we see $M_1$ and $\mathcal{W}^{T}(V)$ are in fact isomorphic as vertex algebras.
		\end{proof}
	\end{theorem}
	
	Let $\mathcal{C}(V)$ be the subalgebra of $\mathcal{W}(V)$ generated by $$L(z)=:\beta(z)\partial\gamma(z):-:b(z)\partial{c}(z):,$$         and 
	$$G(z)=:b(z)\partial\gamma(z):.$$
	We have
	$[Q_{(0)}, G_{(1)}]=L_{(1)}$, with the definitions of $Q, L, G$ agree with Eq. (\ref{eqn:QLJG}).
	It is easy to check that $\mathcal{C}(V) \subset \mathcal{W}^{T}(V)$. 
	
	We conjecture that $\mathcal{C}(V) = \mathcal{W}^{T}(V)$.
	
	The Virasoro field $:\beta\partial\gamma:-:b\partial c: \in \mathcal{W}^{T}(V)$ corresponds naturally to the canonical Virasoro field in $H^0(X, \Omega_X^{ch})$, which exists for any complex manifold. The structure of $\mathcal{W}^{T}(V)$ is independent of the genus $g$ and the choice of complex structure by its construction.
	\subsection{ The space of global sections.}
	
	Let \begin{equation}\label{M2}
		M_2:= \bigoplus_{k}\bigoplus_{l>s} H^0(X, SW(\bar{T}^{*}X))[k, l, s],
		\end{equation}
		  where $H^0(X, SW(\bar{T}^{*}X))[k, l, s]$ is the space of holomorphic sections of an antiholomorphic vector bundle $SW(\bar{T}^{*}X)[k, l, s]$.
	
	Recall that $\mathcal{W}^{T}(V)$ is a sub vertex algebra of $\mathcal{W}(V)$ which consists of $\mathfrak{sl}_2$ invariants.
	
	 The subspace $M_2$ is naturally a module over $M_1$, where $M_1$ is defined by Eq. (\ref{M1}). And by the isomorphism of vertex algebras between $M_1$ and $\mathcal{W}^{T}(V)$, the subspace $M_2$ is a module over $\mathcal{W}^{T}(V)$. $M_2$ is not a vertex algebra.
	 
	 We derive a description of the space of global sections of $\Omega^{ch}_X$ on the closed complex curve $X$ with genus $g\ge2$ from the computations in the preceding parts:
	
	\begin{theorem}\label{thm:global}
		If $X$ is a closed complex curve with genus $g\geq 2$, define spaces $M_1$ and $M_2$ by Eq. (\ref{M1}) and (\ref{M2}), then the space of global sections:
	    \begin{equation}\label{eqn:sum}
	    	H^0(X, \Omega^{ch}_X) \cong M_1\oplus M_2,
	    \end{equation}
	    where $M_1 \cong \mathcal{W}^T(V)$ as vertex algebras. $M_1$ is a sub vertex algebra of $H^0(X, \Omega^{ch}_X)$ and $M_2$ is a module over $\mathcal{W}^T(V)$. 
	\end{theorem}
	
	\begin{definition}\label{def:weightspace}
		Let
	$$H^0(X,\Omega_X^{ch})[k ,l]:=\left\{a \in H^0(X, \Omega^{ch}_X)| L_{(1)}a=ka, J_{(0)}a=la \right\},$$
	$$H^0(X,\Omega_X^{ch})[k ,l, l]:=H^0(X,\Omega_X^{ch})[k ,l]\cap H^0(X, SW(\bar{T}^{*}X))[k, l, l],$$
		$$H^0(X, \Omega^{ch}_X)[k, l, \le s]:=\big\{ a\in H^0(X, \Omega^{ch}_X)[k, l]| \text{$a$ is a section of }\bigoplus_{i\ge 0} SW(\bar{T}^{*}X)[k, l, s-i]\big\},$$
		$$H^0_{+}(X, \Omega^{ch}_X)[k, l]:=H^0(X, \Omega^{ch}_X)[k, l, \le (l-1)].$$
		\end{definition}
		
		By Theorem \ref{thm:global}: $$H^0(X, \Omega^{ch}_X)[k, l]=H^0_{+}(X, \Omega^{ch}_X)[k, l]\oplus H^0(X, \Omega^{ch}_X)[k, l, l].$$ 
		
		We show a method of calculating $\dim H^0_{+}(X, \Omega^{ch}_X)[k, l]$ based on Theorem \ref{thm:global}.
		
			For any $k\geq 0$ and $l$, $s$ such that $l-s>0$, we have
		\begin{equation}\label{eq:dim}
			\dim H^0(X, \Omega^{ch}_X)[k, l, \le s]=\sum_{i\geq 0}\dim H^0(X, SW(\bar{T}^{*}X))[k, l, s-i],
		\end{equation}
		which follows from the exact sequences holding for any $s^{\prime}<l$:
		$$0\rightarrow H^0(X, \Omega^{ch}_X)[k,l, \le(s^{\prime}-1)]\rightarrow H^0(X, \Omega^{ch}_X)[k,l, \le s^{\prime}] \rightarrow H^0(X, SW(\bar{T}^{*}X))[k, l, s^{\prime}]\rightarrow 0.$$
		
		In particular, 	\begin{equation} \label{plus}
			\begin{aligned}
				\dim H^0_{+}(X, \Omega^{ch}_X)[k, l]  &=\dim H^0(X, \Omega^{ch}_X)[k,l, \le (l-1)] \\
				&=\sum_{i> 0}\dim H^0(X, SW(\bar{T}^{*}X))[k, l, l-i],
			\end{aligned}
		\end{equation}
		and $\dim H^0(X, SW(\bar{T}^{*}X))[k, l, l-i]$ is calculated by
			 $$\dim H^0(X, SW(\bar{T}^{*}X))[k, l, s]=\dim H^0(X, \bar{T}X^{\otimes n}\otimes \bar{T}^{*}X^{\otimes m})\dim SW^{\gamma}[k, l, s],$$where $n-m=l-s$.
			 
      In fact, using the Riemann-Roch formula (cf. [\cite{Ha}, IV, Theorem 1.3]), one can derive:
      	\[
      \dim H^0(X, \bar{T}X^{\otimes n}\otimes \bar{T}^{*}X^{\otimes m})=
      \begin{cases}
      	0, & n-m \le -1, \\
      	1, & n-m=0,  \\
      	g, & n-m=1, \\    
      	(n-m)(2g-2)+1-g, & n-m\ge 2.
      \end{cases}
      \]
      
      Let 
         	\begin{equation}\label{eqn:expand}
         		\begin{aligned}
         	H(q, y, z)&:=\prod^{\infty}_{n=1}(1-q^{n-1}y)(1-q^{n}y^{-1})\frac{1}{(1+q^{n}z)}\frac{1}{(1+q^{n}z^{-1})} \\
         	&=\sum_{m, m^{\prime} \in \mathbb{Z}, n\geq 0}a_{n,m,m^{\prime}}(-y)^m(-z)^{m^{\prime}}q^n,
         	\end{aligned}
         \end{equation}
      then   $\dim SW^{\gamma}[k, l, s]=a_{k,l,s}$.
	
	The dimensions of the first few weight spaces can be calculated directly. 
	
	If $k=0$, then $$\dim H^0(X, \Omega^{ch}_X)[0,1]=g,$$ and the dimensions of the other weight spaces are $0$.
	
	If $k=1$, then 
	\begin{equation*}
		\begin{aligned}
			\dim H^0(X, \Omega^{ch}_X)[1,0]&=g,    \\
			\dim H^0(X, \Omega^{ch}_X)[1,1]&=4g-3,  \\
			\dim H^0(X, \Omega^{ch}_X)[1,2]&=3g-3, 
		\end{aligned}
	\end{equation*}
	and the dimensions of the other weight spaces are $0$.
	
	If $k=2$, then
	\begin{equation*}
		\begin{aligned}
			\dim H^0(X, \Omega^{ch}_X)[2,-1]&=1, \\
			\dim H^0(X, \Omega^{ch}_X)[2,0]&=5g-2,    \\
			\dim H^0(X, \Omega^{ch}_X)[2,1]&=14g-11,  \\
			\dim H^0(X, \Omega^{ch}_X)[2,2]&=9g-8, 
		\end{aligned}
	\end{equation*}
	and the dimensions of the other weight spaces are $0$.
	
	We have seen that the dimensions of weight spaces depend on the genus $g$.

\end{document}